\documentclass[11pt]{article}
\usepackage{amssymb,amsthm,amsfonts,amssymb,amsmath,graphicx,latexsym}
\usepackage{mathrsfs,color}
\usepackage[latin1]{inputenc}
\numberwithin{equation}{section}

\DeclareMathOperator{\diam}{diam}
\DeclareMathOperator{\dist}{dist}
\DeclareMathOperator{\dive}{div}

\DeclareMathOperator{\Lip}{Lip}

\newcommand{\ds}{\displaystyle}
\newcommand{\R}{\mathbb R}
\newcommand{\N}{\mathbb N}

\newcommand{\W}{\mathcal W}
\newcommand{\ve}{\varepsilon}

\newcommand{\deb}{\rightharpoonup}

\newcommand{\sL}{\mathscr{L}}

\newcommand{\si}{\sigma}
\newcommand{\ga}{\gamma}
\newcommand{\Om}{\Omega}
\newcommand{\B}{{\cal B}}

\newcommand{\weak}{\rightharpoonup}
\newcommand{\cC}{\mathcal{C}}

\newcommand{\m}{\mathfrak{m}}

\def\A{{\mathcal A}}

\def\F{{\mathcal F}}

\def\M{{\mathcal M}}
\def\P{{\mathcal P}}

\def\H{{\mathcal H}}

\theoremstyle{plain}
\newtheorem{thm}{Theorem}[section]
\newtheorem{prop}[thm]{Proposition}
\newtheorem{lem}[thm]{Lemma}

\newtheorem{defi}[thm]{Definition}
\theoremstyle{remark}
\newtheorem{rem}[thm]{Remark}

\newtheorem*{ack}{Acknowledgements}

\title{A Benamou-Brenier approach to branched transport}

\author{Lorenzo Brasco\textsuperscript{1,2}
\qquad\qquad Giuseppe Buttazzo\textsuperscript{1}
\qquad\qquad Filippo Santambrogio\textsuperscript{2}
}
\date{31.03.2010}

\begin{document}

\maketitle

\footnotetext[1]{\scriptsize\ Dipartimento di Matematica, Universit\`a di Pisa, Largo B. Pontecorvo 5 - 56127 Pisa, ITALY
\texttt{brasco@mail.dm.unipi.it, buttazzo@dm.unipi.it}}
\footnotetext[2]{\scriptsize\ CEREMADE, UMR CNRS 7534, Universit\'e Paris-Dauphine, Pl. de Lattre de Tassigny - 75775 Paris Cedex 16, FRANCE
\texttt{filippo@ceremade.dauphine.fr}}
\bigskip

\begin{abstract}
The problem of branched transportation aims to describe the movement of masses when, due to concavity effects, they have the interest to travel together as much as possible, because the cost for a path of length $\ell$ covered by a mass $m$ is proportional to $m^\alpha\ell$ with $0<\alpha<1$. The optimization of this criterion let branched structures appear and is suitable to applications like road systems, blood vessels, river networks\dots Several models have been employed in the literature to present this transport problem, and the present paper looks at a dynamical one, similar to the celebrated Benamou-Brenier formulation of Kantorovitch optimal transport. The movement is represented by a path $\rho_t$ of probabilities, connecting an initial state $\mu_0$ to a final state $\mu_1$, satisfying the continuity equation $\partial_t\rho+\dive_xq=0$ together with a velocity field $v$ (with $q=\rho v$ being the momentum). The transportation cost to be minimized is non-convex and finite on atomic measures: $\int_0^1\big(\int_\Omega\rho^{\alpha-1}|q|\,d\#(x)\big)\,dt$.
\end{abstract}

\medskip\noindent
{\bf AMS Subject Classification (2000):} 49J45, 49Q20, 49Q10, 90B18, 60K30

\bigskip\noindent
{\bf Keywords:} Optimal transport, branched transport, continuity equation, functionals on spaces of measures, optimal networks.

\section{Introduction}\label{sint}

The optimal mass transportation theory consists in the study of transporting a given mass distribution $\mu_0$ on $\Om$ (that we assume to be a compact and convex subset of $\R^d$) 
into a final configuration $\mu_1$, by minimizing the total transportation cost, the latter being suitably defined: clearly, $\mu_0$ and $\mu_1$ are required to satisfy the mass balance condition $\int_\Om d\mu_0=\int_\Om d\mu_1$. From now on, we will assume that they are normalized to be probability measures. The cost for moving a unit mass from a position $x$ to a position $y$ is taken equal to $c(x,y)$, a function a priori given, which determines the nature of the problem and provides the total minimal cost
\begin{equation}\label{mincost}
C(\mu_0,\mu_1)=\min\left\{\int_{\Om\times\Om} c(x,y)\,d\gamma(x,y)\, :\, \gamma\in\Gamma(\mu_0,\mu_1)\right\}
\end{equation}
where $\Gamma(\mu_0,\mu_1)$ is the class of admissible {\it transport plans}, i.e. probabilities on the product space $\Om\times\Om$ having first and second marginals given by $\mu_0$ and $\mu_1$ respectively. The cases $c(x,y)=|x-y|^p$ with $p\ge1$ have been particularly studied, and the cost $C(\mu_0,\mu_1)$ in \eqref{mincost} provides, through the relation
$$
W_p(\mu_0,\mu_1)=\big(C(\mu_0,\mu_1)\big)^{1/p},
$$
the so-called {\it Wasserstein distance} $W_p$ which metrizes the weak* convergence on the space of probabilities $\P(\Om)$. A very wide literature on the subject is available; we simply mention the books \cite{AGS,vi1,vi2} where one can find a complete list of references.

Thanks to the fact that the space $\W_p(\Om)$ of probability measures endowed with these distances turns out to be a {\it geodesic space}, dynamical models for optimal transportation are of particular interest. Being a geodesic space means that the distance between two points is always equal to the infimum of the lengths of the curves connecting these points, and that this infimum is actually a minimum:
\[
W_p(\mu_0,\mu_1)=\min\left\{\int_0^1|\rho'_t|_{W_p}\,dt\, :\, \rho\in\Lip([0,1];\W_p(\Om)),\ \rho_0=\mu_0,\ \rho_1=\mu_1\right\},
\]
where $|\rho'|_{W_p}$ is the {\it metric derivative} of the measure-valued Lipschitz curve $\rho$, defined as (we refer the reader to \cite{AGS}, for more details)
\[
|\rho'_t|_{W_p}=\lim_{h\to 0}\frac{W_p(\rho_{t+h},\rho_t)}{h}.
\]

Since the curves connecting two points of this space are actually curves of measures, they can be described through the so-called {\it continuity equation}: it is well-known (see \cite{AGS}, Theorem 8.3.1) that for every Lipschitz or absolutely continuous curve $\rho_t$ in the space $\W_p(\Om)$ ($p>1$ for simplicity) there exists a map $q$ from $[0,1]$ into the space of vector valued measures, such that $q_t\ll\rho_t$ (hence $q_t=v_t\cdot \rho_t$, $v$ being the velocity vector) which represents the flux $q=\rho v$ and satisfies
\begin{equation}\label{coneq}
\partial_t\rho+\dive_xq=0\ \mbox{ and }\ \|v_t\|_{L^p(\rho_t)}=|\rho'_t|_{W_p},
\end{equation}
(the degenerate case $p=1$ being a little bit more involved, since $q_t\ll\rho_t$ is no more guaranteed and the $L^1-$norm has to be replaced by the mass of the measure $q_t$, see \cite{Am}).

On the other hand, every time that we have a pair $(\rho,q)$ satisfying $\partial_t\rho+\dive_xq=0$ with $q\ll\rho$, so that $q_t=v_t\cdot\rho_t$, we can infer that $|\rho'_t|_{W_p}\le\|v_t\|_{L^p(\rho_t)}$. This means that one can minimize the functional 
\begin{equation}\label{dyncost}
\A_p(\rho,q):=\begin{cases}\int_0^1\Big(\int_{\R^d}|v_t|^p\,d\rho_t\Big)\,dt&\mbox{ if } q\ll\rho\mbox{ and }q_t=v_t\cdot\rho_t,\\
+\infty&\mbox{ otherwise},\end{cases}
\end{equation}
which is nothing but the integral in time of the kinetic energy when $p=2$, and the cost in \eqref{mincost} can be recovered through the equality
$$
C(\mu_0,\mu_1)=\min\left\{\int_0^1\left(\int_{\R^d}\left|\frac{dq_t}{d\rho_t}(x)\right|^p\,d\rho_t(x)\right)\,dt\, :\, \partial_t\rho+\dive_xq=0,\ \rho_0=\mu_0,\ \rho_1=\mu_1\right\}.
$$
The problem above is the one which was proposed by Benamou and Brenier in \cite{bebr} as a dynamical version of optimal transportation. It has the advantage that it is the minimization of a convex functional of $\rho$ and $q$, under linear constraints. 

Other variants of mass transportation problems have been studied and can be expressed in this way by considering in \eqref{dyncost} other {\it convex} functions of the pair $(\rho,q)$. Recently, Dolbeault, Nazaret and Savar\'e introduced in \cite{donasa} new classes of distances over $\P(\R^d)$ based on the minimization of the functional (where $\lambda$ is a given reference measure on $\R^d$ and $\rho$ and $q$ are identified with their densities w.r.t. $\lambda$)
$$
\int_0^1\left(\int_{\R^d}\Phi(\rho,q)\,d\lambda\right)\,dt,\ \ \mbox{ where }\ \Phi(\rho,q)=\frac{|q|^p}{h(\rho)^{p-1}}=\left(\frac{|q|}{h(\rho)}\right)^p\!h(\rho),\ p\ge1,
$$
which are connected to the non-linear mobility continuity equation $\partial_t\rho+\dive_x\big(h(\rho)v\big)=0$ (a treatment of the limiting case $h(\rho)\equiv1$, corresponding to consider $\Phi(\rho,q)=|q|^p$, can also be found in \cite{bre}). If the function $h$ is concave (for example $h(\rho)=\rho^\beta$, with $\beta\in[0,1]$), this problem turns out to be convex as well. The main interest that motivated Dolbeault et al. to the study of these distances lies in the possible applications to diffusion equations of the type of the non-linear mobility continuity equation $\partial_t\rho+\dive_x\big(h(\rho)v\big)=0$ above, where the vector field $v$ depends on $\rho$ in a way such that the equation can be interpreted as a gradient flow of a given functional with respect to these new family of dynamical distances. 
Moreover, the equations of the geodesics (which are similar to a {\it mean-field game system}, see \cite{LL}) and the conditions for these distances to be finite are being studied in \cite{NazSan}. 

In connection with {\it congestion} effects and crowd motion, other models include penalizations on high densities: in \cite{bujiou} the case
$$\Phi(\rho,q)=\frac{|q|^p}{\rho^{p-1}}+c\rho^2\qquad p\ge1,\ c\ge0$$
has been considered as a model for crowd motion in a congested situation (for instance in case of panic). This problem as well is convex. 

A completely different situation occurs in the case opposite to congestion, when {\it concentration} effects are present and the mass has the interest to travel together as much as possible, in order to save part of the cost. This happens very often in many applications, as discovered by Gilbert who in \cite{gil} formulated a mathematical model for the transportation of signals along telephone cables. More recently the Gilbert's model has been refined and considered in the framework of mass transportation, under the name of {\it branched transport}, to emphasize the fact that transport rays may bifurcate. All these models have in common the fact that the cost for a mass $m$ moving on a path of length $\ell$ is proportional to $m^\alpha \ell$ ($0<\alpha<1$, so that $(m_1+m_2)^\alpha<m_1^\alpha+m_2^\alpha$). In \cite{becamo1,becamo2,becamo3,befi,mamoso} for every $0<\alpha<1$ a transportation cost from $\rho_0$ to $\rho_1$ is considered through a suitable use of probabilities defined on spaces of curves in $\Om$, with \cite{mamoso} (the so called {\it irrigation patterns model}) dealing with the case of a single source $\rho_0=\delta_{x_0}$. See Section 4 to have a glance at the details of these models and their formulations. On the other hand, the model of \cite{xi} can be seen as the natural extension of the original Gilbert's model and uses vector measures having prescribed divergence $\rho_0-\rho_1$: these vector measures are the continuous generalization of the finite weighted and oriented graphs that were present in Gilbert's original formulation. 

A first attempt to obtain a dynamical formulation of branched transportation through curves of measures was made in \cite{brbusa}, and later refined in \cite{bra,brsa}: in these papers the starting point is the geodesic formulation of the Wasserstein distance, where the length functional is modified considering an energy of the type 
$$
\int_0^1 g(\rho_t)\,|\rho'_t|_{W_p}\,dt.
$$ 
The weight function $g$ is a local term of the moving mass, forcing the mass to concentrate and thus giving raise to branching phenomena.

These models are not satisfactory yet, because they are in general not equivalent to those by Gilbert, Xia or Bernot-Caselles-Morel. A tentative to perform some modifications in the functionals defined on curves of measures so as to obtain equivalence with the other models has been made in [11], where on the other hand some quite involved distinction between moving mass and still mass has been done. In all of these models, the branched transportation is studied avoiding the Benamou-Brenier approach consisting in the minimization of a suitable cost $\F(\rho,q)$ under the constraint of the continuity equation $\partial_t\rho+\dive_xq=0$, that we believe {\it is the most natural for this kind of problems}. The only approach to dynamical branched transportation using the continuity equation is, as far as we know, the one of \cite{bose}. Yet, to prove semicontinuity and hence existence, even in this model, the problem is reduced to the minimization of a functional of the form
$$\int\theta^\alpha\,d\H^1(x,t)$$
(which is the energy of Xia in \cite{xi}) and the dynamical features are not completely exploited. 

In the present paper we follow a more direct approach: for all pairs $(\rho,q)$ verifying the continuity equation, with $\rho_0=\mu_0$ and $\rho_1=\mu_1$, we define a functional $\F(\rho,q)$ and we show that this functional is both lower semicontinuous and coercive with respect to a suitable convergence on $(\rho,q)$, and this provides directly the existence of an optimal dynamical path. The paper is organized as follows:
\begin{itemize}
\item in Section \ref{sset} we give the precise setting and state the main results;
\item Section \ref{spro} is devoted to the proofs giving the existence of an optimal path $\rho_t$;
\item in Section \ref{sequ} we show that our model is equivalent to the other models of branched transportation available in the literature, comparing it to the traffic plan model of \cite{becamo1}, which is one of the most flexible (and anyway equivalent to the others, as shown in \cite{becamo3}, Chapter 9);
\item in the Appendix we deal with some inequalities involving Wasserstein distances and branched distances, that is distances over the space of probabilities given by the minima of some branched transportation problems. These inequalities have already been studied in \cite{mosa} and \cite{deso}, but some very precise issues concerning $d_\alpha$ and $W_{1/\alpha}$ are very close to the topics of this paper and deserve being treated here. New and simpler proofs are provided.
\end{itemize}

\section{Problem setting and main results}\label{sset}

In this section we fix the notation and state the main results of the paper. In the following $\Omega$ will denote a given subset of $\R^d$, where all the mass dynamics will take place; for the sake of simplicity we assume that $\Omega$ is convex and compact. The space $\P(\Omega)$ of all Borel probabilities on $\Omega$ can then be endowed with the weak* convergence, which is metrized by the Wasserstein distances (see the Introduction). In the following, we will also use the notation $\M(\Om;\R^d)$ to indicate the space of $\R^d$-valued Radon measures over $\Om$, while $\sL^k$ will indicate the $k-$dimensional Lebesgue measure. 
\vskip.3cm
The main objects to be considered will be pairs $(\rho,q)$ with
\begin{equation}\label{chisiamo}
\rho\in C\big([0,1];\P(\Omega)\big),\ q\in L^1\big([0,1];\M(\Omega;\R^d)\big)
\end{equation}
satisfying the {\it continuity equation} formally written as (here $\nu$ stands for the outer normal versor to $\partial\Om$) 
\begin{equation}\label{coneq2}
\left\{\begin{array}{cccc}
\partial_t\rho+\dive_x q&=&0,& \mbox{ in } [0,1]\times \Om\\
q\cdot\nu&=&0,&\mbox{ on }[0,1]\times\partial\Omega,
\end{array}
\right.
\end{equation}
whose precise meaning is given in the sense of distributions, that is 
\begin{equation}\label{weakform}
\ds\int_0^1\Big[\int_\Omega\partial_t\phi(t,x)\,d\rho_t(x)+\int_\Omega D_x\phi(x,t)\cdot dq_t(x)\Big]\,dt=0
\end{equation}
for every smooth function $\phi$ with $\phi(0,x)=\phi(1,x)=0$.

\begin{defi}
We denote by $\mathfrak{D}$ the set of all pairs $(\rho,q)$ satisfying \eqref{chisiamo} and \eqref{weakform}. Moreover, given $\mu_0,\mu_1\in\P(\Omega)$, we define the set $\mathfrak{D}(\mu_0,\mu_1)$ of admissible configurations connecting $\mu_0$ to $\mu_1$ as
\[
\mathfrak{D}(\mu_0,\mu_1)=\big\{(\rho,q)\in\mathfrak{D}\ :\ \rho_0=\mu_0,\ \rho_1=\mu_1\big\}.
\]
\end{defi}
The velocity vector $v$ can be defined as the Radon-Nikodym derivative of the vector measure $q$ with respect to $\rho$:
$$v=\frac{dq}{d\rho}\;.$$
Among all pairs $(\rho,q)\in\mathfrak{D}$ satisfying the continuity equation above, we consider a cost function $\F(\rho,q)$ of the form
\begin{equation}\label{cost}
\F(\rho,q)=\int_0^1 F(\rho_t,q_t)\,dt,
\end{equation}
where $F$ is defined through
$$
F(\rho,q):=\begin{cases}G_\alpha(|v|^{1/\alpha}\cdot\rho)&\mbox{if }q=v\cdot\rho,\\
+\infty&\mbox{if $q$ is not absolutely continuous w.r.t. }\rho\end{cases}
$$
and $G_\alpha$ ($0<\alpha<1$) is a functional defined on measures, of the kind studied by Bouchitt\'e and Buttazzo in \cite{bobu}: $G_\alpha(\lambda)=+\infty$ if $\lambda$ is not purely atomic, while ($\#$ stands for the counting measure)
$$G_\alpha(\lambda)=\int_\Omega|\lambda(\{x\})|^\alpha\,d\#(x)=\sum_{i\in\N}|\lambda_i|^\alpha,\qquad\hbox{if }\lambda=\sum_{i\in\N}\lambda_i\delta_{x_i}.$$
In this way our functional $\F$ becomes
$$
\F(\rho,q)=\int_0^1\Big[\int_\Omega|v_t(x)|\rho_t(\{x\})^\alpha\,d\#(x)\Big]\,dt
=\int_0^1\Big[\sum_{i\in\N}|v_{t,i}|\rho_{t,i}^\alpha\Big]\,dt,
\qquad(\rho,q)\in\mathfrak{D},
$$
and the dynamical model for branched transport we consider is
\begin{equation}\label{minprob}
\B_{\alpha}(\mu_0,\mu_1):=\min_{(\rho,q)\in\mathfrak{D}(\mu_0,\mu_1)}\F(\rho,q).
\end{equation}
Our main goal is to show that the minimization problem \eqref{minprob} above admits a solution. This will be obtained through the direct methods of the calculus of variations, consisting in proving lower semicontinuity and coercivity of the problem under consideration, with respect to a suitable convergence.

\begin{rem}
We point out that the weak* convergence of the pairs $(\rho,q)$ does not directly imply the lower semicontinuity in \eqref{minprob}, since the functional is not jointly convex. On the other hand, if $(\rho^n,q^n)\in\mathfrak{D}$ and we assume
\[
(\rho^n_t,q^n_t)\rightharpoonup(\rho_t,q_t),\ \mbox{ for $\sL^1-$a.e. } t\in [0,1],
\]
then a simple application of Fatou's Lemma would lead to the desired semicontinuity property of $\F$ (because one could prove that $F$ is a lower semicontinuous functional on measures, as a consequence of the semicontinuity of $G_\alpha$ and of the convexity of $(x,y)\mapsto |x|^p/y^{p-1}$). 
\end{rem}

In order to prove in the easiest possible way a semicontinuity result, we will introduce a convergence which is stronger than the weak convergence of measures on $[0,1]\times\Omega$, but weaker than weak convergence for every fixed time $t$. This convergence will be compatible with the compactness we can infer from our variational problem.
%

\begin{defi}
We say that a sequence $(\rho^n,q^n)$ $\tau$-converges to $(\rho,q)$ if
\((\rho^n,q^n)\rightharpoonup (\rho,q)\) in the sense of measures and 
\[
\sup_{n\in\N,\ t\in[0,1]}F(\rho^n_t,q^n_t)<+\infty.
\]
\end{defi}

\begin{thm}\label{teocoe}
Let $(\rho^n,q^n)$ be a sequence such that $\F(\rho^n,q^n)\le C$, then up to a time reparametrization, $(\rho^n,q^n)$ is $\tau$-compact.
\end{thm}

\begin{thm}\label{teosci}
Let $(\rho^n,q^n)\in\mathfrak{D}$ be a sequence 
$\tau$-converging to $(\rho,q)$. Then
\[
\F(\rho,q)\le\liminf_{n\to\infty}\F(\rho^n,q^n).
\]
\end{thm}

As a consequence we obtain the following existence result.

\begin{thm}\label{teoexi}
For every $\mu_0,\mu_1\in\P(\Om)$, the minimization problem \eqref{minprob} admits a solution.
\end{thm}

\begin{rem}
We remark that, for some choices of the data $\mu_0,\mu_1$ and of the exponent $\alpha$, the statement of Theorem \ref{teoexi} could be empty, because the functional $\F$ could be constantly $+\infty$ on every admissible path $(\rho,q)$ joining $\mu_0$ to $\mu_1$. This issue will be solved in Section \ref{sequ}, where the equivalence with other variational models for branched transportation will be proven. Since for these models finiteness of the minima has been widely investigated, we can infer for instance that if $\alpha> 1-1/d$ then every pair $\mu_0$ and $\mu_1$ can be joined by a path of finite energy. On the other hand, if $\alpha\le 1-1/d$, $\mu_0=\delta_{x_0}$ and $\mu_1$ is absolutely continuous w.r.t. $\sL^d$, then there are no finite energy paths connecting them. 
\end{rem}

\section{Proofs}\label{spro}

A preliminary inequality to all the proofs is the following: if $q\ll\rho$, then $q_t=v_t\cdot \rho_t$ and
\begin{equation}\label{Lalpha}
\begin{split}
F(\rho_t,q_t)&=\sum_i\rho_t(\{x_i\})^\alpha|v_t(x_i)|
=\sum_i\left(\rho_t(\{x_i\}) |v_t(x_i)|^{1/\alpha}\right)^\alpha\\
&\ge\left(\sum_i\rho_t(\{x_i\}) |v_t(x_i)|^{1/\alpha}\right)^\alpha=\|v_t\|_{L^{1/\alpha}(\rho_t)},
\end{split}
\end{equation}
due to the sub-additivity of the function $x\mapsto x^\alpha$. This inequality and its consequences will be discussed in the Appendix as well. In particular it also follows
\begin{equation}\label{Lalpha1}
F(\rho_t,q_t)\ge\|v_t\|_{L^{1}(\rho_t)}=|q_t|(\Omega).
\end{equation}

\subsection{Proof of Theorem \ref{teocoe}}
Due to the fact that the functional $\F$ is $1$-homogeneous in the velocity, it is clear that reparametrizations in time do not change the values of $\F$. By reparametrization, we mean replacing a pair $(\rho,q)$ with a new pair $(\tilde \rho,\tilde q)$ of the form $\tilde\rho_t=\rho_{\varphi(t)}$, $\tilde q_t=\varphi'(t)q_{\varphi(t)}$ (which equivalently means that $\tilde q$ is the image measure of $q$ through the inverse of the map $(t,x)\mapsto (\varphi(t),x)$). Thanks to this invariance, if $(\rho^n,q^n)$ is such that $\F(\rho^n,q^n)\le C$, then one can define a new pair $(\tilde\rho^n,\tilde q^n)$, with 
$$
F(\tilde\rho^n_t,\tilde q^n_t)=\F(\tilde\rho^n,\tilde q^n)=\F(\rho^n,q^n)\le C,\ \mbox{ for every } t. 
$$
After that, we only need to prove compactness for the weak convergence of measures on $[0,1]\times\Omega$, a fact which only requires bounds on the mass of $\tilde\rho^n$ and $\tilde q^n$. The bound on $\tilde\rho^n$ is straightforward, since for every $t$ the measure $\tilde\rho^n_t$ is a probability, while for $\tilde q^n$, which is absolutely continuous w.r.t. $\tilde\rho^n$, it is enough to use \eqref{Lalpha1} in order to bound the mass of $q$ by $C$. 

This allows to extract a subsequence $(\tilde\rho^{n_k}_t,\tilde q^{n_k}_t)$ which converges weakly to a pair $(\rho,q)$. The only nontrivial point is that we a priori restricted our attention to pairs $(\rho,q)$ where $\rho\in C\big([0,1];\P(\Omega)\big)$ and $q\in L^1\big([0,1];\M(\Omega;\R^d)\big)$, so that we need to prove that $\rho$ is continuous and that $q$ is of the form $\int q_t\,dt$. Yet, the inequality \eqref{Lalpha} applied to the pairs $(\tilde\rho^n,\tilde q^n)$, proves a uniform bound on the $L^{1/\alpha}$ norm of the velocities, which implies that the curves $\tilde\rho^n$ are uniformly Lipschitz continuous according to the distance $W_{1/\alpha}$, and this property is inherited by the limit measure $\rho$. 

For the decomposition of $q$, just use the inequality \eqref{Lalpha}, thus obtaining a uniform bound on $\|v^n_t\|_{L^{1/\alpha}(\rho^n_t)}$, which a fortiori gives a uniform bound on the Benamou-Brenier functional 
$$
\A_{1/\alpha}(\rho^n,q^n)=\int_0^1\|v^n_t\|_{L^{1/\alpha}(\rho^n_t)}^{1/\alpha}\,dt.
$$ 
This functional being lower semicontinuous, we can deduce the same bound at the limit: this in particular implies that $q$ is absolutely continuous w.r.t. $\rho$, with an $L^p$ density. Since $\rho$ is a measure on $[0,1]\times\Omega$ which is of the form $\int\rho_t\,dt$, the same disintegration will be true for $q$. 

This means that we have actually found an admissible pair $(\rho,q)$ which is the $\tau-$limit of $(\tilde\rho^{n_k}_t,\tilde q^{n_k}_t)$ and the proof is complete.

\subsection{Proof of Theorem \ref{teosci}}
We consider here a sequence $(\rho^n,q^n)$, where $q^n=v^n\cdot\rho^n$ (otherwise the functional $\F$ would not be finite-valued), satisfying the continuity equation and such that $(\rho^n,q^n)$ $\tau-$converges to $(\rho,q)$.

First of all we define a sequence of measures $\mathfrak{m}^n$ on $[0,1]\times \Omega$ through 
$$ 
\m^n=\int\left(\sum_i\rho^n_t(\{x_{i,t}\})^\alpha |v^n_t(x_{i,t})|\delta_{x_{i,t}}\right)\,dt,
$$
where the points $x_{i,t}$ are the atoms of $q^n_t$ (i.e. the atoms of $\rho^n_t$ where the velocity $v^n_t$ does not vanish). We notice that $\F(\rho^n,q^n)=\m^n([0,1]\times\Omega)$. 

In order to prove lower semicontinuity of $\F$ we can assume $\F(\rho^n,q^n)$ to be bounded. This bound implies the convergence $\m^n\weak\m$, up to the extraction of a subsequence (not relabeled). It is clear that, on this subsequence, we have 
$$
\lim_{n\to\infty}\F(\rho^n,q^n)=\lim_{n\to\infty} \m^n([0,1]\times\Om)=\m([0,1]\times\Omega),
$$ 
then in order to prove the desired semicontinuity property, it is enough to get some proper lower bounds on $\m$. 

Notice that, since we have $\m^n([a,b]\times\Omega)=\int_a^b F(\rho^n_t,q^n_t)\,dt$ and $F(\rho^n_t,q^n_t)\le C$ (by definition of $\tau-$convergence), we know that the marginal of $\m^n$ on the time variable is a measure with an $L^\infty$ density bounded by the constant $C$ on $[0,1]$. This bound is uniform and will be then satisfied by the limit measure $\m$ too: in particular, this implies that we can write $\m=\int\m_t\,dt$.

Let us fix a closed set $Q$, as well as a time interval $[a,b]$, and take the function 
$$
\chi_M(x):=(1-M\,\dist(x,Q))_+,\ x\in\Om,
$$ 
where $(\,\cdot\,)_+$ stands for the positive part: observe that $\chi_M$ is positive, takes the value $1$ on $Q$, is $M$-Lipschitz and vanishes outside a $1/M$-neighborhood of $Q$.
Indicating by $1_E$ the characteristic function of a generic set $E$, i.e. the function which takes the value $1$ on $E$ and $0$ elsewhere, we consider $\phi(t,x)=\chi_M(x)^\alpha 1_{[a,b]}(t)$, which is upper semicontinuous on $[0,1]\times\Omega$. Then we have
\[
\begin{split}
\int\phi(t,x)\,d\m(t,x)&\ge\limsup_{n\to\infty}\int\phi(t,x)\,d\m^n(t,x)\\
&=\limsup_{n\to\infty}\int_a^b\left(\sum_i\rho^n_t(\{x_i\})^\alpha|v^n_t(x_i)|\chi_M(x_i)^\alpha\right)\,dt,
\end{split}
\]
where the points $x_i$ are, as before, the atoms of $q^n$ (and we omitted the dependence on $n$ and $t$). We then decompose the product $\rho^n_t(\{x_i\})^\alpha\chi_M(x_i)^\alpha$ as $\big(\rho^n_t(\{x_i\})\chi_M(x_i)\big)^{\alpha-1}\cdot \big(\rho^n_t(\{x_i\})\chi_M(x_i)\big)$ (where $\rho^n_t\chi_M>0$). Notice that $\rho^n_t(\{x_i\})\chi_M(x_i)\le\int\chi_M\,d\rho^n_t$. Then we can estimate the right-hand side in the previous inequality as
\[
\begin{split}
\int_a^b\left(\sum_i\rho^n_t(\{x_i\})^\alpha|v^n_t(x_i)|\chi_M(x_i)^\alpha\right)\,dt
&\ge\int_a^b\left(\int\chi_M\,d\rho^n_t\right)^{\alpha-1}\!\!\left(\sum_i\rho^n_t(\{x_i\})|v^n_t(x_i)|\chi_M(x_i)\right)\,dt\\
&=\int_a^b\left(\int\chi_M\,d\rho^n_t\right)^{\alpha-1}\left(\int\!\!\chi_M\,d|q^n_t|\right)\,dt.
\end{split}
\]
We go on by estimating from above $\int\chi_M\,d\rho^n_t$: we have
$$
\int_\Om\chi_M(x)\,d\rho^n_t(x)\le\int_\Om\chi_M(x)\,d\rho^n_s(x)+M\,W_1(\rho^n_t,\rho^n_s),
$$
which is a consequence of the definition of $W_1$ by duality with $1$-Lipschitz functions (see \cite{vi1}, Theorem 1.14). To estimate the $W_1$ distance we use $W_1\le W_{1/\alpha}$ and the following fact 
\[
W_{1/\alpha}(\rho^n_t,\rho^n_s)\le\int_s^t|(\rho^n_z)'|_{w_{1/\alpha}}\,dz\le\int_s^t\|v^n_z\|_{L^{1/\alpha}(\rho_z)}\,dz,
\]
then applying inequality \eqref{Lalpha} we have in the end
$$ 
\int_\Om\chi_M(x)\,d\rho^n_t(x)\le\int_\Om\chi_M(x)\,d\rho^n_a(x)+CM(b-a),\ \mbox{ for every }t\in[a,b].
$$
In this way we have 
\[
\begin{split}
\int_a^b\left(\int\chi_M\,d\rho^n_t\right)^{\alpha-1}\!\!\left(\int\chi_M d|q^n_t|\,dt\right)\,dt
&\ge\left(\int\chi_M\,d\rho^n_a\!+\!CM(b-a)\!\!\right)^{\alpha-1}\!\!\!\int_a^b\left(\int\!\!\chi_M\,d|q^n_t|\right)\,dt\\
&=\left(\int\chi_M\,d\rho^n_a+CM(b-a)\right)^{\alpha-1}\int\phi^{1/\alpha}\,d|q^n|.
\end{split}
\]
Hence, we may go on with
\[
\begin{split}
\int\phi\,d\m&\ge\limsup_{n\to\infty}\left[\left(\int\chi_M\,d\rho^n_a+CM(b-a)\right)^{\alpha-1}\int\phi^{1/\alpha}\,d|q^n|\right]\\
&\ge\left(\int\chi_M\,d\rho_a\!+\!CM(b-a)\right)^{\alpha-1}\int\phi^{1/\alpha}\,d|q|.
\end{split}
\]
In the last inequality the second factor has been dealt with in the following way: suppose $|q^n|\deb\sigma$, then we have $\sigma\ge|q|$; moreover $\phi\ge\tilde\phi$ where $\tilde\phi(t,x):=\chi_M(x) 1_{(a,b)}(t)$, and this last function is l.s.c. and positive, so that 
$$
\liminf_{n\to\infty}\int\phi^{1/\alpha}\,d|q^n|
\ge\liminf_{n\to\infty}\int\tilde\phi^{1/\alpha}\,d|q^n|
\ge\int\tilde\phi^{1/\alpha}\,d\sigma
\ge\int\tilde\phi^{1/\alpha}\,d|q|
=\int\phi^{1/\alpha}\,d|q|,
$$ 
since the boundaries $t=a$ and $t=b$ are negligible for $|q|$.

After that, we can divide by $(b-a)$ (keeping for a while $M$ fixed) and pass to the limit as $b\to a$. This gives, for $\sL^1-$a.e. $a\in [0,1]$, 
$$
\int\chi_M(x)^\alpha\,d\m_a(x)\ge\left(\int\chi_M(x)\,d\rho_a(x)\right)^{\alpha-1}\int\chi_M(x)\,d|q_a|(x).
$$
We let now $M\to\infty$, so that $\chi_M$ monotonically converges to the characteristic function of the set $Q$, and we have, by dominated convergence w.r.t. $\m_a, \rho_a$ and $|q_a|$,
\begin{equation}\label{forQ}
\m_a(Q)\ge\rho_a(Q)^{\alpha-1}|q_a|(Q).
\end{equation}
In the last term the convention $0\cdot \infty=0$ is used (if $|q_a|(Q)=0$). This inequality is proven for closed sets, but by regularity of the measures it is not difficult to prove it for arbitrary sets. Actually, if $S\subset\Omega$ is an arbitrary Borel set, we can write
$$ 
\m_a(S)\ge\m_a(Q)\ge\rho_a(Q)^{\alpha-1}|q_a|(Q)\ge\rho_a(S)^{\alpha-1}|q_a|(Q),
$$
for every $Q\subset S$ closed, and take a sequence of closed sets $Q_k$ such that $|q_a|(Q_k)\to|q_a|(S)$, since $|q_a|$ is, for $\sL^1-$a.e. $a\in[0,1]$, a finite (and hence regular) measure on the compact set $\Om$. We want now to prove that
\begin{itemize}
\item $q\ll\rho$,
\item $q_a=v_a\cdot\rho_a$ is atomic for $\sL^1-$a.e. $a\in [0,1]$ (i.e. $\rho_a$ is atomic on $\{v_a\neq 0\}$)
\item $\m_a(\Omega)\ge F(\rho_a,q_a)$, for $\sL^1-$a.e. $a\in [0,1]$.
\end{itemize}
This would conclude the proof.

The first statement follows from the inequality \eqref{Lalpha}, and the behavior of the Benamou-Brenier functional $\A_{1/\alpha}(\rho^n,q^n)$ as in the proof of Theorem \ref{teocoe}, which guarantees an $L^{1/\alpha}$ density of $q$ w.r.t. $\rho$. As a consequence, since $\rho$ is a measure on $[0,1]\times\Omega$ which disintegrates w.r.t. the Lebesgue measure on $[0,1]$, the same will be true for $q$ and we can write $q_t=v_t\cdot\rho_t$.

For the second statement, take the inequality $\m_a(S)\ge\rho_a(S)^{\alpha-1}|q_a|(S)$ which is valid for any Borel set $S$, and apply it to sets which are contained in the Borel set $V_\ve:=\{x\in\Omega\,:\,|v_a(x)|>\ve\}$. For those sets, we have easily $\m_a(S)\ge\ve\rho_a(S)^\alpha$. This means that the measure $\lambda:=\ve^{1/\alpha}\rho_a\llcorner V_\ve$ satisfies the inequality $\lambda(S)^\alpha\le\m(S)$ for every Borel set $S\subset V_\ve$. Since $\m$ is a finite measure, this implies that $\lambda$ is atomic (see Lemma \ref{atomic} below). 
If the same is performed for every $\ve=1/k$, this proves that $\rho_a$ is purely atomic on the set $\{x\, :\, |v_a(x)|\neq 0\}$, that is $q_a=v_a\cdot\rho_a$ is purely atomic.

Once we know that $q_a$ is atomic we infer that $F(\rho_a,q_a)=\sum_i\rho_a(\{x_i\})^{\alpha-1}|q_a|(\{x_i\})$ and we only need to consider $Q=\{x_i\}$ in \eqref{forQ} and add up:
$$ 
\m_a(\Omega)\ge\sum_i\m_a(\{x_i\})\ge F(\rho_a,q_a),
$$
which finally concludes the proof.

\begin{lem}\label{atomic}
Take two finite positive measures $\lambda$ and $\mu$ on a domain $\Omega$, and $\alpha\in(0,1)$. Suppose that the inequality $\lambda(S)^\alpha\le\mu(S)$ is satisfied for every Borel set $S\subset\Omega$. Then $\lambda$ is purely atomic. 
\end{lem}

\begin{proof}
Consider a regular grid on $\Omega$ of step $1/k$, for $k\in\N$, and build a measure $\lambda_k$ by putting, in every cell of the grid, all the mass of $\lambda$ in a single point of the cell. This measure $\lambda_k$ is atomic and we have 
$$
G_\alpha(\lambda_k)=\sum_i\lambda(S_i)^\alpha\le\sum_i\mu(S_i)=\mu(\Omega)<+\infty,
$$
where the $S_i$ are the cells of the grid. If we let $k$ goes to $\infty$, the step of the grid goes to zero and we obviously have $\lambda_k\deb\lambda$. On the other hand, the functional $G_\alpha$ is lower semicontinuous (see \cite{bobu}) and this implies $G_\alpha(\lambda)\le\liminf_{k\to\infty}G_\alpha(\lambda_k)\le\mu(\Omega)<+\infty.$ In particular, $\lambda$ is atomic, thus proving the assertion.
\end{proof}

\subsection{Proof of Theorem \ref{teoexi}}
In order to prove existence, one only needs to take a minimizing sequence and apply Theorem \ref{teocoe} to get a new minimizing sequence which is $\tau-$converging: this new sequence is obtained through reparametrization (which does not change the value of $\F$) and by extracting a subsequence. Since the constraints in the problem are linear, i.e. $\rho_i=\mu_i$ for $i=0,1$ and the continuity equation, the limit $(\rho,q)$ will satisfy the same constraints as well. The semicontinuity proven in Theorem \ref{teosci} allows to obtain the existence of a solution.

\section{Equivalence with previous models}\label{sequ}

In this section we prove the equivalence of problem \eqref{minprob} with the other previous formulations of branched transport problems, existing in literature. In particular, as a reference model we will take the one presented in \cite{becamo3}, in which the energy is defined as
\[
E_\alpha(Q)=\int_{\cC}\int_0^1|\si(t)|^{\alpha-1}_Q\,|\si'(t)|\,dt\,dQ(\si),
\]
where $\cC=C([0,1];\Om)$, $Q$ is probability measure over $\cC$ and concentrated on the set $\mathrm{Lip}([0,1];\Om)$ ({\it traffic plan})
and for every $x\in\Om$, the quantity $|x|_Q$ is the {\it multiplicity} of $x$ with respect to $Q$, defined by
\[
|x|_Q=Q\left(\left\{\si\in\cC\ :\ x\in\si([0,1])\right\}\right).
\]
Given $\mu_0,\mu_1\in\P(\Om)$, the corresponding minimum problem is then given by
\[
d_\alpha(\mu_0,\mu_1)=\min_{Q\in TP(\mu_0,\mu_1)}E_\alpha(Q),
\]
where $TP(\mu_0,\mu_1)$ is the set of traffic plans with prescribed time marginals at $t=0,1$, that is
\[
TP(\mu_0,\mu_1)=\{Q\in\cC\, :\, Q \mbox{ concentrated on } \mathrm{Lip}([0,1];\Om),\, (e_i)_\sharp Q=\mu_i,\ i=0,1\},
\]
and $e_t:\cC\to\Om$ is the {\it evaluation map} at time $t$, given by $e_t(\si)=\si(t)$ for every $\si\in\cC$.

\begin{rem}\label{xiarem}
We point out that this model is completely equivalent to the one developed by Xia (see \cite{xi} for the presentation of the model and \cite{becamo3}, Chapter 9, for the equivalence), which is based on a relaxation procedure, starting from an energy defined on finitely atomic probability measures $\mu_0$ and $\mu_1$. In particular, thanks to this relaxed formulation, we get that for every $\mu_0$ and $\mu_1$, there exist two sequences $\mu^n_0$ and $\mu^n_1$ of finitely atomic probability measures, weakly converging to $\mu_0$ and $\mu_1$ respectively, and such that 
\begin{equation}\label{approx}
d_\alpha(\mu^n_0,\mu^n_1)\to d_\alpha(\mu_0,\mu_1).
\end{equation}
\end{rem}

We also need to consider a slight modification of the functional $E_\alpha$ above, introduced in \cite{befi}, 
\[
C_\alpha(Q)=\int_\cC\int_0^1|(\si(t),t)|^{\alpha-1}_Q\,|\si'(t)|\,dt\,dQ(\si),
\]
where now the {\it synchronized multiplicity} $|(x,t)|_Q$ is 
\[
|(x,t)|_Q=Q(\{\si\in\cC\,:\, \si(t)=x\}).
\] 
This second multiplicity accounts for the quantity of curves passing at the same time through the same point, while the one used in the definition of $E_\alpha$ considered all the curves passing eventually through the point: in this sense, the model corresponding to the energy $C_\alpha$ is more dynamical in spirit.
 As a straightforward consequence of the definition of the two multiplicities, we get
\begin{equation}\label{befiineq}
|\si(t)|_Q\ge |(\si(t),t)|_Q, 
\end{equation}
so that $E_\alpha(Q)\le C_\alpha(Q)$. 
Concerning the comparison between the minimization of $E_\alpha$ and $C_\alpha$, we recall the following result (see \cite{befi}, Theorem 5.1).

\begin{thm}\label{berfig}
Let $\mu_0,\mu_1\in\P(\Om)$, with $\mu_0$ a finite sum of Dirac masses. Then for every $\alpha\in[0,1]$ we get
\[
\min_{Q\in TP(\mu_0,\mu_1)}E_\alpha(Q)=\min_{Q\in TP(\mu_0,\mu_1)}C_\alpha(Q).
\]
\end{thm}

We are now in a position to state and prove a result giving the equivalence between our model and the one relative to the energy $E_\alpha$.

\begin{thm}
For every $\alpha\in(0,1)$ and $\mu_0,\mu_1\in\P(\Om)$ we get
\begin{equation}\label{minima}
\B_{\alpha}(\mu_0,\mu_1)=\min_{Q\in TP(\mu_0,\mu_1)}E_\alpha(Q)=d_\alpha(\mu_0,\mu_1).
\end{equation}
\end{thm}

\begin{proof}
We first prove the inequality $\B_{\alpha}(\mu_0,\mu_1)\ge d_\alpha(\mu_0,\mu_1)$. Clearly, if $\B_{\alpha}(\mu_0,\mu_1)=+\infty$ there is nothing to prove; otherwise, take $(\rho,q)$ optimal, which implies, by the way, that $q=v\cdot \rho$ and that $q$ is atomic. Thanks to the {\it superposition principle} (see \cite{AGS}, Theorem 8.2.1) we can construct a probability measure $Q\in\cC$ such that $\rho_t=(e_t)_\sharp Q$ and $Q$ is concentrated on absolutely continuous integral curves of $v$, in the sense that
\[
\int_C\left|\si(t)-\si(0)-\int_0^t v_s(\si(s))\,ds\right|\,dQ(\si)=0,\ \mbox{ for every }t\in[0,1].
\]
Using this information, together with the fact that $E_\alpha\le C_\alpha$ and exchanging the order of integration, we get
\[
\begin{split}
E_\alpha(Q)\le C_\alpha(Q)&=\int_{\cC}\int_0^1|(\si(t),t)|_Q^{\alpha-1}\,|\si'(t)|\,dt\,dQ(\si)\\
&=\int_0^1\int_{\cC}|(\si(t),t)|_Q^{\alpha-1}\,|\si'(t)|\,dQ(\si)\,dt\\
&=\int_0^1\int_{\cC}|(\si(t),t)|_Q^{\alpha-1}\,|v_t(\si(t))|\,dQ(\si)\,dt\\
&=\int_0^1\int_{\Om}|(x,t)|_Q^{\alpha-1}\,|v_t(x)|\,d\rho_t(x)\,dt.
\end{split}
\]
Then we observe that, by virtue of the fact that $\rho_t=(e_t)_\sharp Q$, there holds
\[
|(x,t)|_Q=Q(\{\si\in\cC\, :\, \si(t)=x\})=\rho_t(\{x\}),
\]
so that we can rewrite the last integral as
\[
\int_0^1\int_\Om\rho_t(\{x\})^{\alpha-1}\,|v_t(x)|\,d\rho_t(x)\,dt
=\int_0^1\int\rho_{t}(\{x\})^{\alpha-1}\,d|q_t|(x)\,dt
=\int_0^1\sum_{i\in\N}|v_{t,i}|\rho_{t,i}^\alpha\,dt,
\]
which then gives
\[
d_\alpha(\mu_0,\mu_1)=\min_{Q\in TP(\mu_0,\mu_1)}E_\alpha(Q)\le\F(\rho,q)=\B_{\alpha}(\mu_0,\mu_1).
\]
In order to prove the reverse inequality, we first prove that
\begin{equation}\label{minore}
\B_{\alpha}(\mu_0,\mu_1)\le\min_{Q\in TP(\mu_0,\mu_1)}C_\alpha(Q).
\end{equation}
Take $Q\in TP(\mu_0,\mu_1)$ optimal for $C_\alpha$, then we know that there exists a pair $(\rho,q)$ which is a solution of the continuity equation, with $\rho_t=(e_t)_\sharp Q$ and $q_t=v_t\cdot\rho_t$. The velocity $v$ may be chosen as
\[
v_t(x)=\int\si'(t)\,dQ^{t,x}(\si),
\] 
where $Q^{t,x}$ is the disintegration of $Q$ with respect to the evaluation function $e_t$ (see \cite{Li} for this representation formula of the velocity field $v$). This means that each $Q^{t,x}$ is a probability measure concentrated on the set $\{\si\in\cC\, :\, \si(t)=x\}$ and $Q=\int Q^{t,x}\,d\rho_t(x)$. Therefore, arguing as before
\[
\begin{split}
C_\alpha(Q) &=\int_0^1\int_\cC|(\si(t),t)|^{\alpha-1}_Q\,|\si'(t)|\,dQ(\si)\,dt\\
&=\int_0^1\int_\Om|(x,t)|^{\alpha-1}_Q\left(\int|\si'(t)|\,dQ^{t,x}(\si)\right)\,d\rho_t(x)\,dt\\
&\ge\int_0^1\int_\Om|(x,t)|^{\alpha-1}_Q\,|v_t(x)|\,d\rho_t(x)\,dt\\
&=\int_0^1\int_\Om\rho_t(\{x\})^{\alpha-1}_Q\,|v_t(x)|\,d\rho_t(x)\,dt,
\end{split}
\]
that gives the desired inequality \eqref{minore} since, even if we do not know that $q_t$ or $\rho_t$ are atomic we can restrict the last integral to the set of atoms of $\rho$.
\par
Summarizing, up to now we have shown
\[
d_\alpha(\mu_0,\mu_1)\le \B_{\alpha}(\mu_0,\mu_1)\le \min_{Q\in TP(\mu_0,\mu_1)}C_\alpha(Q),
\]
and equality holds whenever $\mu_0$ is a finite sum of Dirac masses, thanks to Theorem \ref{berfig}.
In order to conclude, it is enough to notice that
thanks to Remark \ref{xiarem}, we may take two sequences $\mu^n_0$ and $\mu^n_1$ of finitely atomic probability measures such that $\mu^n_0\weak \mu_0$, $\mu^n_1\weak \mu_1$ and 
\[
d_\alpha(\mu^n_0,\mu^n_1)\to d_\alpha(\mu_0,\mu_1),
\]
thus getting
\[
d_\alpha(\mu_0,\mu_1)\le \B_{\alpha}(\mu_0,\mu_1)\le \liminf_{n\to\infty} \B_{\alpha}(\mu^n_0,\mu^n_1)\le\lim_{n\to\infty} d_\alpha(\mu^n_0,\mu^n_1)=d_\alpha(\mu_0,\mu_1), 
\] 
hence concluding the proof.
\end{proof}

\begin{rem}
Observe that in the previous Theorem, we did not only prove the equality of the minima, but we also provided a natural way to pass from a minimizer of our formulation {\it \`a la} Benamou-Brenier to a minimizer of the traffic plans model and back. The two problems are thus equivalent in the sense that they describe the same kind of energy and the same optimal structures of branched transport: the simple equality of the minima \eqref{minima} is just a consequence of this more important fact. 
\end{rem}

\section*{Appendix: the distances $d_\alpha$ and $W_{1/\alpha}$}\label{swas}

This last section is devoted to estimates between the distance $d_\alpha$ induced by the branched transport and the Wasserstein distances $W_p$. In particular, in \cite{mosa} the following estimates are proven for $\alpha>1-1/d$ and $p\ge1$:
\[
d_\alpha \le C\, W_{p}^{d(\alpha-1)+1}.
\]
As far as lower bounds on $d_\alpha$ are concerned, the most trivial one is $d_\alpha\ge W_1$ but \cite{deso}, Theorem 8.1, also proves $d_\alpha \ge W_{1/\alpha},$ which is slightly better. Moreover, for scaling reasons (w.r.t. the mass) it is not possible to go beyond $p=1/\alpha$ in this last inequality.

In this paper we already needed to estimate some branched transport cost in terms of $W_{1/\alpha}$ distances and metric derivatives. In this section we prove the inequalities
$$W_{1/\alpha}\le d_\alpha\le C\,W_{1/\alpha}^{d(\alpha-1)+1}\qquad\forall\alpha\in(1-1/d,1].$$
These inequalities are just a particular case of those that are already known. We restrict our attention to $p=1/\alpha$ in the first one and $p\ge 1/\alpha$ in the second one (as we said, they are proven in \cite{deso,mosa}), but the proof we will provide is different and somehow simpler.

The first inequality will be approached through the formulation of branched transport we gave in this paper, but the main tool (i.e. inequality \eqref{Lalpha}) is essentially in common with \cite{deso} and \cite{maso}. What is different is the way to extend this idea to generic measures, i.e. non-atomic ones.

\begin{thm}
For every $\mu_0,\mu_1\in\P(\Om)$ we get
\begin{equation}\label{devisoli}
W_{1/\alpha}(\mu_0,\mu_1)\le d_\alpha(\mu_0,\mu_1).
\end{equation} 
\end{thm}

\begin{proof}
We first observe that thanks to the results of the previous section, for every $\mu_0,\mu_1\in\P(\Om)$ we get
\[
d_\alpha(\mu_0,\mu_1)=\int_0^1\Big[\int_\Omega|v_t(x)|\rho_t(\{x\})^\alpha\,d\#(x)\Big]\,dt.
\] 
for a suitable $(\rho,q)$ admissible in the formulation \eqref{minprob}, with $q=\rho v$. Moreover, using once more the inequality \eqref{Lalpha} the right-hand side in the previous expression can be estimated as
\[
\int_0^1\left[\int_\Omega|v_t(x)|\rho_t(\{x\})^\alpha\,d\#(x)\right]\,dt\ge\int_0^1\|v_t\|_{L^{1/\alpha}(\rho_t)}\,dt
\]
and finally, using the fact that $(\rho,q)$ is solution of the continuity equation, we can infer (see \cite{AGS}, Theorem 8.3.1)
\[
|\rho'_t|_{W_{1/\alpha}}\le\|v_t\|_{L^{1/\alpha}(\rho_t)},\ \mbox{ for $\sL^1$-a.e. }t\in [0,1],
\]
so that
\[
d_\alpha(\mu_0,\mu_1)\ge\int_0^1|\rho'_t|_{W_{1/\alpha}}\,dt\ge W_{1/\alpha}(\mu_0,\mu_1),
\]
where in the last inequality we just estimated the length of a curve by the distance between its endpoints. Thus we have obtained \eqref{devisoli}, concluding the proof.
\end{proof}

In order to prove the other inequality, first of all we have to introduce some notations: we set $Q=[0,1)^d$ and $Q_L=[0,L)^d$, for every $j\in\N$ we consider the following subset of multi-indexes
\[
B_j=\{z\in\N^d\ :\ \|z\|_\infty\le2^j-1\},
\]
and observe that $\#(B_j)=2^{jd}$, then we make a partition of the cube $Q_L$ by dyadic cubes having edge length $L/2^j$, i.e.
\[
Q_L=\bigcup_{i=1}^{2^{jd}}Q^i_{j}:=\bigcup_{z\in B_j}\frac{L\,Q+L\,z}{2^j}.
\]
For every $\mu\in\P(\Om)$ such that $\Om\subset Q_L$, its {\it dyadic approximation} is given by
\[
a_j(\mu)=\sum_{i=1}^{2^{jd}}\mu^i_{j}\,\delta_{x^i_{j}},
\]
where $\mu^i_j=\mu(Q^i_j)$ and $x^i_j$ is the center of $Q^i_j$. We shall always assume that $\Om\subset Q_L$ for a suitable $L$, then the following estimate is well-known (Proposition 6.6, \cite{becamo3}).

\begin{prop}\label{prop:dyadic}
Let $\alpha\in(1-1/d,1]$, then for every $\mu\in\P(\Om)$ we have
\begin{equation}\label{dyadic}
d_\alpha(a_j(\mu),\mu)\le\frac{2^{(d(1-\alpha)-1)j}}{2^{1-d(1-\alpha)}-1}\,\frac{L\sqrt{d}}{2}. 
\end{equation}
\end{prop}

The main tool if one wans to estimate $d_\alpha$ from above by a power of
$W_p$ is to show that the distance $d_\alpha$ between two dyadic approximations can be estimated in terms of their $1/\alpha$-Wasserstein distance: this is the content of the next result.

\begin{lem} 
Let $\alpha\in(1-1/d,1]$, then for every $\mu_0,\mu_1\in\P(\Om)$ we get
\begin{equation}\label{atomics}
d_\alpha(a_j(\mu_0),a_j(\mu_1))\le C\, W_{1/\alpha}(a_j(\mu_0),a_j(\mu_1))\, 2^{jd(1-\alpha)},
\end{equation}
with $C$ depending only on $N$ and $\alpha$.
\end{lem}

\begin{proof}
Let us consider an optimal transport $\ga_j$ between $a_j(\mu_0)$ and $a_j(\mu_1)$, for the cost $c(x,y)=|x-y|^{1/\alpha}$, that is $\ga_j\in\P(\Om\times\Om)$ and it is of the form
\[
\ga_j=\sum_{i,k=1}^{2^{jd}} M_j(i,k) \delta_{x^i_j}\otimes\delta_{x^k_j},
\]
with the $2^{jd}\times 2^{jd}$ matrix $\{M_j(i,k)\}_{i,k}$ belonging to the convex set $\mathfrak{M}$ given by
\[
\mathfrak{M}=\left\{ \{a_{i,k}\}_{i,k}\,:\, a_{i,k}\ge 0,\ \sum_{i=1}^{2^{jd}} a_{i,k}=\mu_1(Q^k_j),\ \sum_{k=1}^{2^{jd}} a_{i,k}=\mu_0(Q^i_j) \right\}.
\]
We know by optimality that $\{M_j(i,k)\}_{i,k}$ can be taken to belong to $\mathrm{Ext\,}(\mathfrak{M})$, the set of extremal points of $\mathfrak{M}$, which consists of the so-called {\it acyclic matrices} (see \cite{De}). They are those matrices belonging to $\mathfrak{M}$ such that the following property holds:
\[
\prod_{r=1}^s a_{i_r k_r}a_{i_{r} k_{r+1}}=0,
\]
for every $2\le s\le2^{jd}$ and every set of indices $i_1<\dots<i_s\in\{1,\dots,2^{jd}\}$, $k_1<\dots<k_j\in\{1,\dots,2^{jd}\}$ (the convention $i_{2^{jd}+1}=i_1$ and $k_{2^{jd}+1}=k_1$ is used).
This implies in particular that
\begin{equation}\label{acyclic}
\#\{(i,k)\ :\ M_j(i,k)\not= 0\}\le 2\cdot 2^{jd},
\end{equation}
that is $\{M_j(i,k)\}_{i,k}$ has at most $2\cdot 2^{jd}$ non-zero entries: in other terms, this optimal transport plan $\ga_j$ does not move more than $2\cdot 2^{jd}$ atoms. Setting $|x^i_j-x^k_j|=\ell_{i,k}$, we then get
\[
W_{1/\alpha}(a_j(\mu_0),a_j(\mu_1))=\left(\sum_{i,k=1}^{2^{jd}}M_j(i,k)\,\ell_{i,k}^{1/\alpha}\right)^\alpha,
\] 
and using \eqref{acyclic} and Jensen's inequality 
\[
\begin{split}
d_\alpha(a_j(\mu_0),a_j(\mu_1))&\le\sum_{i,k=1}^{2^{jd}} M_j(i,k)^\alpha \ell_{i,k}=\sum_{i,k=1}^{2^{jd}}\left(M_j(i,k)\, \ell_{i,k}^{\frac{1}{\alpha}}\right)^{\alpha}\\
&\le \left(\sum_{i,k=1}^{2^{jd}}M_j(i,k)\,\ell_{i,k}^\frac{1}{\alpha}\right)^\alpha\left(\#\{(i,k)\, :\, M_j(i,k)\not= 0\}\right)^{1-\alpha}\\
&\le C\,W_{1/\alpha}(a_j(\mu_0),a_j(\mu_1))\, 2^{jd(1-\alpha)}, 
\end{split}
\]
concluding the proof.
\end{proof}

\begin{thm}
Let $\alpha\in(1-1/d,1]$, then for every $p\ge 1/\alpha$, we get
\begin{equation}\label{xia_wasser}
d_\alpha(\mu_0,\mu_1)\le C\, W_p(\mu_0,\mu_1)^{d(\alpha-1)+1},
\end{equation}
with a constant $C$ depending only on $d$, $\alpha$ and on the diameter of $\Om$.
\end{thm}

\begin{proof}
It is enough to show the validity of \eqref{xia_wasser} for $p=1/\alpha$, then the general case will be just a consequence of the
monotonicity property of the Wasserstein distances, i.e. 
\[
W_{1/\alpha}\le W_p,\ \mbox{ for every } p\ge 1/\alpha.
\]
Using the triangular inequality, \eqref{dyadic} and \eqref{atomics}, we get for every $j\in\N$
\[
\begin{split}
d_\alpha(\mu_0,\mu_1)&\le d_\alpha(\mu_0,a_j(\mu_0))+d_\alpha(a_j(\mu_0),a_j(\mu_1))+d_\alpha(a_j(\mu_1),\mu_1)\\
&\le C\,2^{(d(1-\alpha)-1)j}+d_\alpha(a_j(\mu_0),a_j(\mu_1))\\
&\le C\,2^{(d(1-\alpha)-1)j}+C\,W_{1/\alpha}(a_j(\mu_0),a_j(\mu_1))\, 2^{jd(1-\alpha)},
\end{split}
\]
and
\[
\begin{split}
W_{1/\alpha}(a_j(\mu_0),a_j(\mu_1))&\le W_{1/\alpha}(a_j(\mu_0),\mu_0)+W_{1/\alpha}(\mu_0,\mu_1)+W_{1/\alpha}(a_j(\mu_0),\mu_1)\\
&\le C\, 2^{-j}+W_{1/\alpha}(\mu_0,\mu_1),
\end{split}
\]
which finally gives
\[
\begin{split}
d_\alpha(\mu_0,\mu_1)&\le C\,2^{(d(1-\alpha)-1)j}+C\,W_{1/\alpha}(\mu_0,\mu_1)\, 2^{jd(1-\alpha)}\\
&= C\,2^{(d(1-\alpha)-1)j}\left(1+W_{1/\alpha}(\mu_0,\mu_1)\,2^{j}\right).
\end{split}
\]
It is now sufficient to choose the index $j$ in such a way that
\[
\frac{\diam(\Om)} {2^j}\le W_{1/\alpha}(\mu_0,\mu_1)\le\frac{\diam(\Om)}{2^{j-1}},
\]
which in turn yields
\[
2^{(d(1-\alpha)-1)j}(1+W_{1/\alpha}(\mu_0,\mu_1)2^j)\le C\, W_{1/\alpha}(\mu_0,\mu_1)^{d(\alpha-1)+1},
\]
thus giving the thesis.
\end{proof}

\begin{rem}
As we briefly mentioned, observe that the distances $d_\alpha$ and $W_{1/\alpha}$ have exactly the same scaling with respect to the mass.
\end{rem}

\begin{rem}
We point out that the very same $\mu_0$ and $\mu_1$ of Example 6.19 in \cite{becamo3} show that the exponent $d(\alpha-1)+1$ in inequality \eqref{xia_wasser} cannot be improved.
\end{rem}

\begin{ack}
The authors acknowledge the support of {\it Agence Nationale de la Recherche} via the research project OTARIE, of the Universit\'e Franco-Italienne via the mobility program {\it Galil\'ee} ``Allocation et Exploitation et Evolution Optimales des Ressources: r\'eseaux, points et densit\'es, mod\`eles discrets et continus'' as well of the Universit\`a di Pisa through the program {\it Cooperazione Accademica Internazionale} ``Optimal trasportation and related topics''. The first author has been partially supported by the European Research Council under FP7, Advanced Grant n. 226234 ``Analytic Techniques for Geometric and Functional Inequalities''.
\end{ack}

\end{document}